\documentclass[a4paper]{amsart}
\usepackage{amsmath,amssymb,amsthm,xcolor}
\usepackage[hidelinks]{hyperref}

\usepackage{tikz-cd}
\usepackage[margin=3.4cm]{geometry}
\usepackage[normalem]{ulem}

\newtheorem{thm}{Theorem}[section]
\newtheorem{prop}[thm]{Proposition}
\newtheorem{lem}[thm]{Lemma}

\newtheorem{cor}[thm]{Corollary}
\theoremstyle{definition}
\newtheorem{defn}[thm]{Definition}
\newtheorem{exam}[thm]{Example}
\newtheorem{rem}[thm]{Remark}

\newcommand{\Isom}{\operatorname{Isom}}

\newcommand{\Aut}{\operatorname{Aut}}

\newcommand{\OG}{\mathsf{O}}
\newcommand{\dS}{\mathsf{dS}}
\newcommand{\GL}{\operatorname{GL}}

\newcommand{\Ad}{\operatorname{Ad}}
\newcommand{\ad}{\operatorname{ad}}

\newcommand{\Der}{\operatorname{Der}}

\newcommand{\bT}{\mathbb{T}}
\newcommand{\bR}{\mathbb{R}}

\newcommand{\bN}{\mathbb{N}}
\newcommand{\bZ}{\mathbb{Z}}

\title[Isometries of spacetimes without observer horizons]{Isometries of spacetimes \\ without observer horizons}
\date{\today}
\author{Leonardo Garc\'ia-Heveling}
\address{Fachbereich Mathematik, Universit\"at Hamburg, Germany}
\curraddr{Facult\"at f\"ur Mathematik, Universit\"at Wien, Austria}
\email{leonardo.garcia.heveling@univie.ac.at}
\urladdr{https://leogarciaheveling.github.io/}
\author{Abdelghani Zeghib}
\address{UMPA, CNRS, \'Ecole Normale Sup\'erieure de Lyon, France}
\email{abdelghani.zeghib@ens-lyon.fr}
\urladdr{http://www.umpa.ens-lyon.fr/~zeghib/}

\thanks{We thank the organizers of the conference ``Contact and Lorentzian Geometry II'' at the Ruhr-Universit\"at Bochum, where this collaboration started. LGH thanks Eric Ling for interesting conversations about cosmology. This work was supported in part by the Austrian Science Fund (FWF) [Grant DOI 10.55776/EFP6]. For open access purposes, the authors have applied a CC-BY public copyright license to any author accepted manuscript version arising from this submission.}

\begin{document}

\begin{abstract}
 We study the isometry groups of (non-compact) Lorentzian manifolds with well-behaved causal structure, aka causal spacetimes satisfying the ``no observer horizons'' condition. Our main result is that the group of time orientation-preserving isometries acts properly on the spacetime. As corollaries, we obtain the existence of an invariant Cauchy temporal function, and a splitting of the isometry group into a compact subgroup and a subgroup roughly corresponding to time translations. The latter can only be the trivial group, $\mathbb{Z}$, or $\mathbb{R}$.
\end{abstract}

\maketitle

\tableofcontents

\section{Introduction}

The celebrated theorem of Myers and Steenrod \cite{MyeSte} states that the isometry group $\Isom(M,g)$ of a Riemannian manifold $(M,g)$ is a Lie group, and that if $M$ is compact, then so is $\Isom(M,g)$. While the first part is true for semi-Riemannian manifolds of any signature, the statement about compactness is not. In particular, there exist simple examples of Lorentzian metrics on the torus with non-compact isometry group.

Recall that on a spacetime (a Lorentzian manifold with a time-orientation, i.e.\ a preferred choice of ``future'' light-cones), one has the \emph{causal relation} between points, defined as follows: $p \leq q$ if and only if there exists a future-directed causal curve from $p$ to $q$, meaning that the tangent vector stays in the solid future-lightcone. Time-orientation preserving isometries clearly preserve the causal relation, a fact that we exploit in this paper.

On compact spacetimes, it is known that the causal relation cannot be anti-symmetric, because there always exist closed causal curves. Because of this, spacetimes considered in general relativity are generally non-compact, since closed causal curves would imply the possibility of time-travel, which is considered unphysical. In the present paper, we will see that constraints on the causal relation lead to isometries being quite rigid, because they have to preserve the global causal structure. This contrasts with the defining property of an isometry, i.e.\ preserving the metric tensor, which is merely local.

In addition to causality (i.e.\ antisymmetry of $\leq$), we assume the ``no future observer horizons'' condition (NFOH), which ensures that every observer is able to receive signals from every point in spacetime (after waiting long enough). The precise definition is that $I^-(\gamma) = M$ for every future-inextendible causal curve $\gamma$. This condition has been extensively studied in the context of the causal boundary and Bartnik's splitting conjecture \cite{GalRev}. As pointed out by Penrose \cite{Pen}, the combination of causality and the NFOH implies another standard causality assumption, namely global hyperbolicity with compact Cauchy surfaces. It has been shown by various authors that spacetimes with certain kinds of ``time-translation like'' symmetry must satisfy the NFOH \cite{CSFH,GalSym,GarBoundary}. Here, we turn the situation around, and study how the NFOH restricts the possible symmetries of the spacetime.

\begin{thm} \label{thm:IsomproperIntro}
 Let $(M,g)$ be a causal spacetime satisfying the no future observer horizons condition. Then the group $\Isom^\uparrow(M,g)$ of time orientation preserving isometries acts properly on $M$.
\end{thm}

Under the same assumptions, the following two statements hold.

\begin{cor} \label{cor:timeIntro}
 There exists a Cauchy temporal function $\tau \colon M \to \bR$ such that the action of $\Isom^\uparrow(M,g)$ on $M$ preserves the {differential} $d\tau$.
\end{cor}

\begin{cor} \label{cor:prodintro}
 The isometry group splits as a semi-direct product
 \begin{equation} \label{eq:prodintro}
  \Isom^\uparrow(M,g) = L \ltimes N,
 \end{equation}
 where $N$ is a compact Lie group, and $L$ is either trivial, $\bZ$, or $\bR$. Any function $\tau$ as in Corollary~\ref{cor:timeIntro} is preserved by $N$, while $L$ acts on $\tau$ by translations. In the case that $L = \bR$, the connected component of the identity splits as a direct product
 \begin{equation*}
  \Isom^{\uparrow, 0}(M,g) = \bR \times N^0.
 \end{equation*}
\end{cor}

Properness of the action (which we recall later on) means essentially the action has a non-chaotic dynamics, so that all orbits are {topologically} closed, and even more, the orbit space is Hausdorff. Corollary~\ref{cor:prodintro} shows that, when the NFOH is satisfied, the splitting of the spacetime into a product of time and space via a Cauchy temporal function (due to Geroch \cite{Ger} and Bernal and S\'anchez \cite{BeSa1}) induces a splitting of the isometry group. The NFOH is essential here; a counterexample is given by de Sitter spacetime, which has compact Cauchy surfaces, but does not satisfy any of our conclusions. See also M\"uller \cite{Mue} for prior work on temporal functions invariant under a Lie group action, where the spacetime is only assumed globally hyperbolic, but the Lie group is assumed compact.

In comparison, in Riemannian signature, the isometry group always acts properly. A largely open question is to classify globally hyperbolic spacetimes having their isometry group acting non-properly. This leads to non-Riemannian behavior and relativistic phenomena, which can already be observed in Minkowski spacetime endowed with the Poincar\'e group action, as well as de Sitter and anti de Sitter spacetimes. There are, however, few works addressing this question, essentially (to our knowledge) due to Monclair on 2-dimensional spatially compact globally hyperbolic spacetimes \cite{Mon1,Mon2}. It is actually in the non-physical case of compact Lorentzian manifolds that the analogous question was intently investigated, e.g.\ by Zimmer \cite{Zim}, Gromov \cite{Gro}, Adams and Stuck \cite{AdSt}, Zeghib \cite{Zeg1,Zeg2}, Piccione and Zeghib \cite{PiZe}, Frances \cite{Fra}... The motivation here comes from the rigidity theory of group actions preserving geometric structures on compact manifolds.

The paper is structured as follows: In Section~\ref{sec:prelim} we discuss preliminaries on Lie group actions, Lorentzian geometry, {and the physical context}. The goal is both to make the paper more accessible, and to prove certain basic results that are not found in the literature in that precise form. Section~\ref{sec:Isomproper} contains the proofs of the results stated in this introduction (restated there with some added detail), and many examples. In Section~\ref{sec:extra}, we have included some additional results and examples. In particular, we prove uniform bi-Lipschitz continuity of certain subgroups of isometries, and give an application to spacetimes equipped with a regular cosmological time function.

\section{Preliminaries} \label{sec:prelim}

\subsection{Lie groups}

We recall some basic facts about Lie groups (see \cite{Knapp} for details). First, we recall the definition of semi-direct product. Let $L, N$ be groups and $\Psi \colon L \to \Aut(N),\ l \mapsto \Psi_l$ a homomorphism from $L$ to the group $\Aut(N)$ of automorphisms of $N$. Then $N \rtimes_\Psi L$ is defined as the Cartesian product $N \times L$ equipped with the operation
\begin{equation*}
 (n,l) \cdot (n',l') = \left(n \Psi_l( n') ,l l'\right).
\end{equation*}
Note that in the statement of Corollary~\ref{cor:prodintro} we have reversed the order of the factors (to better accommodate to the conventions in Lorentzian geometry, where the ``time part'' comes before the ``space part''), and we have omitted the map $\Psi$ (because we are not making any claims about it). The above construction is the ``outer'' semi-direct product. There is also an ``inner'' semi-direct product, defined as a decomposition of a given group $G$ into subgroups, as follows: $N$ is a normal subgroup, and $L$ is a subgroup such that $N \cap L = \{e\}$ and $G = NL$. The latter means that every element $g \in G$ is of the form $g = nl$ for $n \in N$ and $l \in L$ (where $n$ and $l$ are necessarily unique). If we then define
\begin{equation} \label{eq:Psi}
 \Psi \colon L \longrightarrow \Aut(N), \quad l \longmapsto (n \mapsto l n l^{-1}),
\end{equation}
where $l n l^{-1} \in N$ because $N$ was assumed normal, we see that
\begin{equation*}
 nln'l' = n ln'l^{-1} l l' = n \Psi_l( n') l l',
\end{equation*}
implying that $G$ is canonically isomorphic to the outer direct product $N \rtimes_\Psi L$. The map $\Psi$ plays the role of specifying the commutation rules between elements of $N$ and $L$. A direct product is one where $\Psi_l$ is the identity on $N$ for all $l$, i.e., where all elements of $N$ commute with all elements of $L$.

Now let $G$ be a Lie group. Its Lie algebra $\mathfrak{g}$ is the vector space of smooth left-invariant vector fields on $G$, equipped with the usual Lie bracket $[\cdot,\cdot]$. There is a canonical isomorphism of vector spaces $\mathfrak{g} \to T_eG$ between the Lie algebra and the tangent space at the identity $e \in G$. The exponential map $\exp \colon \mathfrak{g} \to G$ sends a vector $X \in \mathfrak{g}$ to $\exp(X) := \varphi(1) \in G$, where $\varphi \colon \bR \to G$ is the unique one-parameter subgroup with $\varphi'(0) = X$. A word of caution is appropriate here: We will also use $\exp$ to denote the exponential map on a semi-Riemannian manifold (defined through geodesics); it should be clear from context which one is meant.

The Lie algebra of a semi-direct product $G = N \rtimes_\Psi L$ is of a specific form, called the semi-direct sum of Lie algebras, denoted by $\mathfrak{n} \oplus_{d \Psi} \mathfrak{l}$. Here $d \Psi \colon \mathfrak{l} \to \Der(\mathfrak{n})$ is the differential of $\Psi \colon L \to \Aut(N)$, but the semi-direct sum can also be considered as an abstract construction. As a vector space, $\mathfrak{n} \oplus_{d \Psi} \mathfrak{l}$ is the direct sum of $\mathfrak{n}$ and $\mathfrak{l}$. It is equipped with the bracket that restricts to the respective brackets on $\mathfrak{n}$ and $\mathfrak{l}$, and satisfies
\begin{equation*}
 [X,Y] = d \Psi(X)(Y) \text{ for } X \in \mathfrak{l}, \ Y \in \mathfrak{n}.
\end{equation*}
A somewhat related concept is the adjoint representation $\Ad \colon G \to \GL(\mathfrak{g})$, which assigns to $g \in G$ the differential $\Ad(g) := d \psi_g \colon \mathfrak{g} \to \mathfrak{g}$ of the map
\begin{equation*}
 \psi_g \colon G \longrightarrow G, \quad x \longmapsto g x g^{-1},
\end{equation*}
cf.\ \eqref{eq:Psi}. It then holds that
\begin{equation*}
 \exp(\Ad(g)X) = g \exp(X) g^{-1}.
\end{equation*}

\subsection{Group actions}

An action of a group $G$ on a space $M$ is a map $G \times M \to M,\ (g,p) \mapsto g(p)$ such that for all $p \in M$, $g(h(p)) = (gh)(p)$ for all $g,h \in G$ and $e(p) = p$ for $e \in G$ the identity element. Let us start with a standard proposition.

\begin{prop}[{\cite[Prop.~21.5]{Lee}}] \label{prop:properactions}
 Let $M$ be a smooth manifold and $G$ a Lie group acting continuously on $M$. The following are equivalent:
 \begin{enumerate}
  \item The action is proper, meaning that $G \times M \to M \times M, (g,p) \mapsto (g(p),p)$ is a proper map.
  \item For every pair of sequences $(p_i)_i$ in $M$ and $(g_i)_i$ in $G$ such that $(p_i)_i$ and $(g_i(p_i))_i$ converge in $M$, $(g_i)_i$ has a convergent subsequence in $G$.
  \item For every compact set $K \subseteq M$, the set $G_K := \{ g \in G \mid K \cap g(K) \neq \emptyset \}$ is compact.
 \end{enumerate}

\end{prop}

For the rest of this subsection, let $M$ be a smooth manifold, and let $g$ be a semi-Riemannian metric on $M$ of signature $(k,l)$. The orthonormal frame bundle $O(M)$ consists of all pairs of a point $p \in M$ and an ordered $g$-orthonormal basis of $T_pM$ (a \emph{frame}). One can also view a frame at the point $p \in M$ as a linear isometry $F_p \colon \bR^{k,l} \to (T_pM,g_p)$, which necessarily maps the standard basis to an orthonormal basis. This way, one has a canonical action of the indefinite orthogonal group $\OG(k,l)$ on $O(M)$ by mapping $F_p$ to $F_p \circ \phi$ for $\phi \in \OG(k,l)$, equipping $O(M)$ with the structure of an $\OG(k,l)$-principal bundle. In this section, we fix
\begin{equation*}
 n := \dim(M) \quad \text{and} \quad \bar n := \dim(\OG(k,l)) = \frac{n(n-1)}{2},
\end{equation*}
and therefore $\dim(O(M)) = n + \bar{n}$.

{Theorem \ref{Thm:OM} below, telling us that the isometry group acts properly on $O(M)$, will be a central tool in many of our proofs. We are not aware of a reference where it appears in this exact form, but similar ideas can be found in the book of Kobayashi \cite[Chap.~II.1]{KobBook} and, with more details, in the lecture notes by Ballmann \cite{BallNotes}. In those references, it is proven that the isometry group of a (semi-)Riemannian manifold is a Lie group. Before we start, we need the following well-known lemma.}

\begin{lem} \label{lem:parallel}
 The orthonormal frame bundle admits a parallelization, that is, a diffeomorphism
 \begin{equation*}
  \Phi \colon O(M) \times \bR^{\bar n + n} \longrightarrow TO(M)
 \end{equation*}
 such that $\Phi(F_p,\cdot) \colon \bR^{\bar n + n} \longrightarrow T_{F_p}O(M)$ is an isomorphism of vector spaces. Moreover, $\Phi$ can be chosen invariant under isometries of $(M,g)$, meaning that if $\psi \in \Isom(M,g)$ then
 \begin{equation} \label{eq:presPhi}
  \Phi(D_p\psi(F_p),z) = D\psi ( \Phi(F_p,z)) \quad \text{for all } z \in \bR^{\bar n + n}.
 \end{equation}
\end{lem}

Here $D_p \psi$ denotes the differential of $\psi$ at the point $p \in M$, and $D \psi$ the differential as a map on the tangent bundle $TM$. Note that $F_p \mapsto \Phi(F_p,z)$ defines a vector field on $O(M)$. Such vector fields are called \emph{constant}. The automorphism group $\Aut(\Phi)$ is defined as the group of diffeomorphisms of $M$ satisfying \eqref{eq:presPhi}. Hence the last statement of the lemma is equivalent to stating that $\Isom(M,g)$ is a subgroup of $\Aut(\Phi)$.

\begin{proof}
 We follow \cite[Sec.~3]{BallNotes}. Let $\pi \colon O(M) \to M$ denote the canonical projection. Our first goal is to split the tangent bundle $TO(M)$ into a vertical and a horizontal part, $TO(M) = V \oplus H$. Here $V = \ker D\pi$, and we want $D\pi \colon H_{F_p} \to T_pM$ to be an isomorphism of vector spaces. While $V$ is uniquely defined, $H$ is not, and our task is to find a such an $H$ that is compatible with our geometric structure (i.e.\ with the metric $g$).

 For $F_p \in O(M)$ and $z \in \bR^n$, we set $v := F_p(z) \in T_pM$. Choose a curve $c \colon (-\epsilon,\epsilon) \to M$ with $c(0) = p$ and $c' (0) = v$. Let $E \colon (-\epsilon,\epsilon) \to O(M)$ be the unique parallel (with respect to the Levi-Civita connection on $(M,g)$) frame along $c$ such that $E(0) = F_p$. We claim that the map
 \begin{equation*}
  \Phi_H \colon O(M) \times \bR^n \longrightarrow TO(M), \qquad (F_p,z) \longmapsto E'(0).
 \end{equation*}
 is well-defined (i.e.\ independent of the choice of $c$), smooth, and linear in $z$. This can be verified by choosing coordinates around $p$ such that the Christoffel symbols vanish at $p$, and extending them to a local trivialization of $O(M)$. In such coordinates, $E'(0) = (v^1,...,v^{p+q},0,...,0)$.

 The image $H$ of $\Phi_H$ defines the desired horizontal distribution. This is so, because $D\pi(\phi(F_p,z)) = F_p(z)$, and hence $\Phi_H(F_p, \cdot) \colon \bR^n \to T_{F_p} O(M)$ is an injection. Moreover, when restricting the target space to $H$, we obtain a trivialization
 \begin{equation*}
  \Phi_H \colon O(M) \times \bR^n \longrightarrow H,
 \end{equation*}
 which we denote again by $\Phi_H$. The action of $O(k,l)$ on $O(M)$ also provides a trivialization for the vertical distribution $V$, namely
 \begin{equation*}
  \Phi_V \colon O(M) \times \mathfrak{o}(k,l) \longrightarrow V, \qquad (F_p,x) \longmapsto \partial_t \vert_{t=0} (F_p \circ \exp(t x)).
 \end{equation*}
 After choosing an identification of the Lie algebra $\mathfrak{o}(k,l)$ with $\bR^{\bar{n}}$, we obtain the desired parallelization by setting
 \begin{equation*}
  \Phi \colon O(M) \times \bR^{\bar{n}+n} \longrightarrow TO(M), \quad (F_p,(x,z)) \longmapsto \Phi_V(F_p,x) + \Phi_H(F_p,z).
 \end{equation*}
 By construction, $\Phi$ is invariant (i.e.\ satisfies \eqref{eq:presPhi}) under $D\psi \colon O(M) \to O(M)$ for any isometry $\psi \in \Isom(M,g)$ (since $D\psi$ preserves the bundle structure of $O(M)$ and the Levi-Civita connection).
\end{proof}

\begin{thm} \label{Thm:OM}
 The isometry group $\Isom(M,g)$ of a semi-Riemannian manifold $(M,g)$ acts freely and properly on the orthonormal frame bundle $O(M)$.
\end{thm}

\begin{proof}
 Let $\Phi$ be the parallelization of $O(M)$ obtained in Lemma \ref{lem:parallel}. By \cite[Prop.~1.5]{BallNotes}, the action of the automorphism group $\Aut(\Phi)$ is free, and the orbit map of any given $F_p \in O(M)$,
 \begin{equation*}
  \Theta_{F_p} \colon \Aut(\Phi) \longrightarrow O(M), \qquad \psi \longmapsto D_p\psi(F_p),
 \end{equation*}
 is a proper map. We want to prove that, in fact, the action of $\Aut(\Phi)$ is proper. Then also the action of $\Isom(M,g)$ is proper, because $\Isom(M,g)$ is a closed subgroup of $\Aut(\Phi)$. The latter is true if $\psi_i \to \psi$ in the $C^1$-topology, then $D \psi_i \to D \psi$, and thus $\psi_i^* g = g$ implies $\psi^* g = g$. In words, the limit of a sequence of isometries is again an isometry. Note that the $C^0$ and $C^1$ topologies on $\Isom(M,g)$ coincide, because the exponential map is a local diffeomorphism, and hence convergence of $\psi_i$ on a normal neighborhood of a point $p \in M$ is equivalent to convergence of $D_p \psi_i$ on $T_pM$.

 Notice that the parallelization $\Phi$ allows us to define an invariant Riemannian metric $h$ on $O(M)$ as follows. Let $(e_i)_i$ be the standard basis of $\bR^{\bar n + n}$. We declare that $(\Psi(F_p,e_i))_i$ is an orthonormal basis for $h$ at $F_p \in O(M)$. By \eqref{eq:presPhi}, $h$ is invariant under $\Aut(\Phi)$.

 The rest of the proof consists in arguing that an action that preserves a Riemannian metric and has proper orbit maps is proper. Suppose that $F_i \in O(M)$, $\phi_i \in \Aut(\Phi)$ are sequences such that $F_i \to F$ and $\phi_i (F_i) \to \tilde F$ converge in $O(M)$. We need to show that $\phi_i$ has a convergent subsequence in $\Aut(\Phi)$ (see Proposition \ref{prop:properactions}). Choose $r>0$ be small enough so that the $h$-balls $B_r(F)$ and $B_{4r}(\tilde F)$ are compact. Let $i_0 > 0$ be large enough such that $d(F_i,F) < r$ and $d(\phi_i(F_i),\tilde F) < r$ for all $i \geq i_0$, where $d$ is the distance induced by $h$. Then
 \begin{equation*}
  \phi_i(F) \in \phi_i(B_r(F)) \subseteq \phi_i(B_{2r}(F_i)) = B_{2r}(\phi_i (F_i)) \subseteq B_{4r}(\tilde F).
 \end{equation*}
 Hence the sequence $\phi_i$ is contained in $\Theta_{F}^{-1}(B_{4r}(\tilde F))$, and the latter set is compact by properness of the orbit map $\Theta_F$. It follows that $\phi_i$ has a convergent subsequence, and hence the action is proper.
\end{proof}

\subsection{Lorentzian geometry}

We start with a brief summary of definitions relevant to this paper; a more complete introduction can be found in books such as \cite{BEE,ONeill}.

A (smooth) \emph{spacetime} is a smooth connected Lorentzian manifold $(M,g)$ together with a \emph{time orientation}, i.e.\ a smooth vector field $T$ on $M$ such that $g(T,T) < 0$ everywhere (i.e.\ $T$ is everywhere timelike). Vectors $X \in TM$ that are timelike or causal (meaning $g(X,X) < 0$ or $g(X,X) \leq 0$, respectively) are then called \emph{future-directed} (f.d.) if $g(X,T) < 0$. This then gives rise to notions of f.d.\ timelike (or causal) curve $\gamma \colon I \to M$ by requiring that $\gamma$ is Lipschitz and its tangent vector $\dot\gamma$ is f.d.\ timelike (or causal) almost everywhere. This, in turn, defines the chronological and causal relations $\ll$ and $\leq$ on $M$ via: $p \ll q$ if there is a f.d.\ timelike curve from $p$ to $q$, and $p \leq q$ if there is a f.d.\ causal curve from $p$ to $q$ or $p=q$. We say that $(M,g)$ is causal if $\leq$ is antisymmetric. It is a standard fact that $p \ll q$ is an open condition in $M \times M$. We denote the timelike and causal futures and pasts of a point $p \in M$ by
\begin{align*}
 I^+(p) &= \{q \in M \mid p \ll q\},&  J^+(p) &= \{q \in M \mid p \leq q\},&\\
 I^-(p) &= \{q \in M \mid p \gg q\},&  J^-(p) &= \{q \in M \mid p \geq q\}.
\end{align*}
For subsets $A \in M$, we set $I^+(A) := \bigcup_{a \in A} I^+(a)$, etc. For curves $\gamma \colon I \to M$, we denote $I^+(\gamma) := I^+(\gamma(I))$.

\begin{defn}
 Let $(M,g)$ be a smooth spacetime.
 \begin{enumerate}
  \item A (smooth, spacelike) \emph{Cauchy surface} is a smooth hypersurface $\Sigma \subset M$ such that $g \vert_{T\Sigma}$ is positive definite and every inextendible causal curve in $(M,g)$ intersects $\Sigma$ exactly once.
  \item A \emph{time function} $\tau \colon M \to \bR$ is a continuous function that is strictly increasing along all f.d.\ causal curves.
  \item A \emph{temporal function} $\tau \colon M \to \bR$ is a smooth function such that $d\tau(X) > 0$ for every f.d.\ causal $X \in TM$. It is called \emph{Cauchy temporal function} if the level set $\tau^{-1}(t)$ is a Cauchy surface, for every $t \in \bR$.
 \end{enumerate}
\end{defn}

Notice that temporal functions form a special class of time functions. The following theorem is central. It summarizes results of Geroch \cite{Ger} and Bernal and S\'anchez \cite{BeSa1,BeSa2}.

\begin{thm} \label{thm:split}
 Let $(M,g)$ be a spacetime. The following are equivalent:
 \begin{enumerate}
  \item $(M,g)$ is globally hyperbolic, meaning that $\leq$ is antisymmetric and for all $p,q \in M$, $J^-(p) \cap J^+(q)$ is compact.
  \item There is a Cauchy surface $\Sigma \subset M$.
  \item There is a Cauchy temporal function $\tau \colon M \to \bR$.
 \end{enumerate}
 Moreover, any two Cauchy surfaces are diffeomorphic, and every Cauchy temporal function $\tau$ gives rise to a diffeomorphism $M \to \bR \times \Sigma$ such that $\tau$ coincides with the {first} projection.
\end{thm}

Thanks to the above, a spacetime having compact Cauchy surfaces is a well-defined notion, which we indirectly assume throughout the paper, as a consequence of our main hypothesis below.

\begin{defn}
 A spacetime $(M,g)$ is said to satisfy:
 \begin{enumerate}
  \item The \emph{No Future Observer Horizons condition (NFOH)} if $I^-(\gamma) = M$ for every future-inextendible causal curve $\gamma$.
  \item The \emph{No Past Observer Horizons condition (NPOH)} if $I^+(\gamma) = M$ for every past-inextendible causal curve $\gamma$.
  \item The \emph{No Observer Horizons condition (NOH)} if it satisfies NFOH and NPOH.
 \end{enumerate}
\end{defn}

It is equivalent to require the same conditions only on timelike curves {\cite[Prop.~3.32]{FHS}}. Moreover, the NFOH (NPOH) is equivalent to the future (past) causal boundary consisting of a single point. We formulate most results for the NFOH, but there is always an analogous version for the NPOH. We first prove that the NFOH upgrades one of the lowest levels on the causal ladder, causality, to the highest level, global hyperbolicity (see also \cite[p.~624]{Pen} and \cite[Sec.~9]{Flo}).

\begin{prop} \label{prop:NFOHCauchy}
 Let $(M,g)$ be a causal spacetime satisfying the NFOH. Then $(M,g)$ is globally hyperbolic with compact Cauchy surfaces.
\end{prop}

\begin{proof}
 {The NFOH is equivalent to the absence of future lightlike rays, i.e.\ future inextendible achronal causal curves. Clearly, if $\gamma$ is future-inextendible, then $I^-(\gamma) = M$ and $\gamma$ cannot be achronal. Conversely, if $I^-(\sigma) \neq M$ for some future-inextendible causal $\sigma$, then $\partial I^-(\sigma)$ is non-empty and is ruled by future inextendible lightlike rays. Minguzzi proved that if $(M,g)$ is chronological and has no future lightlike rays, then it is globally hyperbolic \cite[Thm.~8]{MinCMP}, settling the first part of the proposition. Then Theorem~\ref{thm:split} implies there is a Cauchy surface $\Sigma \subset M$. A result of Wald and Yip \cite[Lem.~2]{WaYi} states that for every future inextendible causal curves $\gamma$, $\Sigma \cap I^-(\gamma)$ is compact, which in our case means that $\Sigma \cap M = \Sigma$ is compact (note that the NFOH implies that the future causal boundary is a singleton, and hence it is spacelike, so the assumptions of \cite[Lem.~2]{WaYi} are met).}
\end{proof}

Lemma~\ref{lem:NOHCauchy} below provides an alternative characterization of the NFOH.

\begin{defn}[{\cite[p.~3]{Paeng}}]
 A pair of subsets $A,B \subset M$ is called \emph{totally timelike connected} if for all points $p \in A$ and $q \in B$ it holds that $p \ll q$.
\end{defn}

Note that the order matters, since in particular, it follows that $B$ is in the future of $A$. Moreover, total timelike connectedness defines a transitive relation on subsets.

\begin{lem} \label{lem:NOHCauchy}
 Suppose $(M,g)$ is globally hyperbolic with compact Cauchy surfaces. Then the following are equivalent:
 \begin{enumerate}
  \item $(M,g)$ satisfies the NFOH.
  \item For every Cauchy temporal function $\tau \colon M \to \bR$ and every $s \in \bR$, there exists a $t > s$ such that $\tau^{-1}(s)$ and $\tau^{-1}(t)$ are totally timelike connected.
 \end{enumerate}
\end{lem}

\begin{proof}
 $(ii) \implies (i)$. Let $\tau \colon M \to \bR$ be a Cauchy temporal function, let $\gamma$ be a future inextendible causal curve, and let $p \in M$ be a point. For $s = \tau(p)$, assumption (ii) guarantees the existence of a $t \in \bR$ such that $p \ll q$ for all $q$ with $\tau(q) = t$. Since $\gamma$ is inextendible, it must intersect the level set $\tau^{-1}(t)$, and hence $p \in I^-(\gamma)$. Since $p$ was arbitrary, $I^-(\gamma) = M$.

 $(i) \implies (ii)$. Let $\tau$ be a Cauchy temporal function, $s \in \bR$, and let $\Sigma$ denote $\tau^{-1}(s)$. By Theorem~\ref{thm:split}, $M$ splits into $\bR \times \Sigma$. By (i), since $ \lambda \mapsto (\lambda,x)$ is an inextendible causal curve for every $x \in \Sigma$, the collection $I^-((\lambda,x)) \cap \Sigma$ for all $\lambda \in \bR$ forms an open cover of $\Sigma$. By compactness of $\Sigma$, there is a finite subcover, and hence a value $T_x$ such that $\Sigma \subset I^-((T_x,x))$. Property (ii) then follows if we can choose $T_x$ independent of $x$.

 To see the latter, let $x \in \Sigma$ be fixed, and notice that if $y \in \Sigma \cap I^-((T_x,x))$, then there exist neighborhoods $U$ of $y$ and $V$ of $(T_x,x)$ such that every point in $U$ is in the chronological past of every point in $V$. By compactness, we may cover $\Sigma$ by finitely many such $U$, and take the intersection $\tilde V$ of the corresponding $V$'s. Let $p \colon \bR \times \Sigma \to \Sigma$ denote the canonical projection. Then, for all $z \in p(\tilde{V})$, we have that $\Sigma \subset I^-((T_x,z))$. Repeating this construction for every $x \in \Sigma$, we obtain an open cover of $\Sigma$ by sets of the form $p(\tilde{V}_i)$, and corresponding values $T_i$ such that $\Sigma \subset I^-((T_i,z))$ for all $z \in p(\tilde{V}_i)$. By compactness, there is a finite subcover $p(\tilde{V}_j)$ for $j = 1,...,J$. Setting $T = \max_j\{T_j\}$ we conclude that $\Sigma$ and $\tau^{-1}(T)$ are totally timelike connected.
\end{proof}

Proposition~\ref{prop:NFOHCauchy} and Lemma~\ref{lem:NOHCauchy} have the following interesting consequence.

\begin{lem} \label{lem:tauCauchy}
 If $(M,g)$ is causal and satisfies the NOH, then every time function $\tau \colon M \to \bR$ has Cauchy level sets.
\end{lem}

\begin{proof}
  {By Proposition~\ref{prop:NFOHCauchy}, $(M,g)$ is globally hyperbolic, and hence admits a Cauchy temporal function $\tilde \tau$. Let $\Sigma_t$ be any level set of $\tau$, and $p \in \Sigma_t$. By Lemma~\ref{lem:NOHCauchy}, we can find level sets $\tilde\Sigma_{t_i}$ of $\tilde \tau$ for values $t_1 < t_2 = \tilde\tau(p) < t_3$ such that $\tilde \Sigma_{t_i}$ and $\tilde \Sigma_{t_{i+1}}$ are totally timelike connected. Because $\Sigma_t$ is achronal and intersects $\tilde \Sigma_{t_2}$, it cannot intersect $I^-(\tilde \Sigma_{t_1})$ or $I^+(\tilde \Sigma_{t_3})$. Thus $\Sigma_t \subset J^+(\tilde \Sigma_{t_1}) \cap J^-(\tilde \Sigma_{t_3})$ and by \cite[Cor.~2]{GalCauchy}, $\Sigma_t$ is a Cauchy surface.}
\end{proof}

{Note that if $(M,g)$ only satisfies the NFOH, the statement is false in general. One can adapt the above proof, however, to show that when the NFOH is satisfied, given any time function $\tau$, all level sets corresponding to values larger than some constant are Cauchy surfaces. For example, on the spacetime $(0,\infty) \times S^1$ with product metric $-dt^2 + d\theta^2$, the level sets of the time function $\tau(t,\theta) = t + 1/2 \sin(\theta)$ are not Cauchy for values of $\tau$ smaller than $1/2$.}

\subsection{Physical context}

Let us conclude this section with a physics oriented discussion of the NFOH and our results. In fact, it is more reasonable to consider the NPOH instead. By Proposition~\ref{prop:NFOHCauchy}, our spacetimes will be globally hyperbolic with compact Cauchy surfaces. This places us in a cosmological context, since spacetimes modelling isolated systems (rather than the whole Universe) are usually asymptotically flat, and hence have non-compact Cauchy surfaces. Even in cosmology, the observed spatial flatness of the Universe has popularized FLRW spacetimes with flat slices (isometric to $\bR^n$, hence non-compact). Recall that an FLRW spacetime is a warped product $((0,\infty) \times \Sigma,-dt^2 + a^2(t) h)$, where $(\Sigma,h)$ is one of the Riemannian model spaces of constant curvature, and $a(t)$ is called the scale factor.

{There are, however, reasons not to discard compact Cauchy surfaces altogether. Recently, the notion that spatial isotropy leads one to consider only FLRW spacetimes has been challenged \cite{San,Ava}. Specifically, S\'anchez has introduced cosmological models which are foliated by spacelike slices of constant curvature, but which are not FLRW spacetimes, and where the sign of the curvature and the topology of the slices can change. In particular, there can be flat (non-compact, non-Cauchy) slices, even if the spacetime is globally hyperbolic with compact Cauchy surfaces.}

{Moving on, recall that for FLRW spacetimes, there is the \emph{no particle horizons} condition
\begin{equation} \label{eq:particle}
  \int_0^1 \frac{1}{a(t)} dt = \infty.
\end{equation}
In the case of compact spatial slices $\Sigma$, this condition is equivalent to the NPOH (see \cite[Prop.~3.6]{Sbi}  for a proof, which is valid for more general warped products, as it is not used that the slices are simply connected and of constant curvature). Current inflationary cosmology stipulates that there are no particle horizons, or at least that they are very large (i.e.\ the integral \eqref{eq:particle} is very large) \cite[Chap.~4]{Wei}. This is because the observed isotropy of the cosmic microwave background is explained by thermalization, which requires different spatial regions to have been in causal contact in the early universe. The characterization of the NPOH given by Lemma~\ref{lem:NOHCauchy} provides a very strong version of this: any two points in space are causally connected in their past.}

{Let us use this opportunity to also discuss the intuitive meaning of the NFOH. One can interpret it as requiring that the spacetime has compact Cauchy surfaces (as is established by Proposition~\ref{prop:NFOHCauchy}) and that additionally, the lightcones do not become excessively narrow towards future infinity (analogously for past infinity when considering the NPOH). Choosing a Cauchy splitting of the metric
\begin{equation*}
 M = \bR \times \Sigma,\quad g = -\beta dt^2 + h = \beta(-dt^2+h/\beta),
\end{equation*}
a sufficient condition for NFOH is
\begin{equation*}
  g_t/\beta \leq h(t)g_0/(\beta|_{t=0})
\end{equation*}
where $h(t)$ cannot diverge too fast when $t \to +\infty$ (this meaning that the cones are not narrowing too fast towards the future). This can be made precise in terms of the \emph{future arrival time function} \cite[Prop.~4.3]{San2}. Using this language, Lemma 2.9 can also be proven via continuity of the future arrival time function \cite[Lem.~4.1]{San2}.}

\section{The action of $\Isom^\uparrow(M,g)$ is proper} \label{sec:Isomproper}

In this section, we prove the main results stated in the introduction, and then provide some related examples. We denote by $\Isom^\uparrow(M,g)$ the group of time-orientation preserving isometries of $(M,g)$, i.e.\ isometries $\phi \colon M \to M$ such that for every future-directed causal vector $X \in TM$, $d\phi(X)$ is again future-directed. By abuse of nomenclature, we will sometimes refer to $\Isom^\uparrow(M,g)$ simply as ``the isometry group''.

\subsection{Proofs of the main results}

We start by proving the main theorem, or, in fact, a slight strengthening of it: we only need the no observer horizons condition in one direction, i.e., the NFOH (or the NPOH).

\begin{thm} \label{thm:Isomproper}
 Let $(M,g)$ be a causal spacetime satisfying the NFOH. Then $\Isom^\uparrow(M,g)$ acts properly on $M$.
\end{thm}

\begin{proof}
 By Proposition~\ref{prop:properactions}, the action is proper if and only if for every pair of sequences $p_i \in M$ and $\phi_i \in \Isom^\uparrow(M,g)$ such that $p_i$ and $q_i := \phi_i(p_i)$ converge in $M$, a subsequence of $\phi_i$ converges in $\Isom^\uparrow(M,g)$. By Theorem \ref{Thm:OM}, we know that this will be the case if there is a converging sequence of frames $F_i$ at $p_i$, for which a subsequence of $D \phi_i (F_i)$ converges in $O(M)$.

 We proceed to show that, in fact, for every converging sequence of frames $F_i$, the sequence $D \phi_i (F_i)$ is contained in a compact set of $O(M)$. Note that for a sequence of Lorentzian orthonormal frames whose basepoints converge, the sequence leaves every compact set if and only if there is a subsequence that becomes degenerate in the limit. By this we mean that the hyperplane spanned by the spacelike vectors of the frame tends to a null plane (equivalently, the timelike vector in the frame tends to a null vector).

 Suppose that we have a converging sequence of frames $F_i$ at $p_i \to p$ such that $D \phi_i (F_i)$ has converging basepoints $q_i$ but becomes degenerate in the limit. This can only happen if there is a sequence of spacelike vectors $v_i$ at $p_i$ that converges to a spacelike $v \neq 0$, such that $w_i := c_i D\phi_i(v_i)$ converges to a future-directed null vector $w$ at $q := \lim_{i \to \infty} \phi_i(p_i)$. Here $c_i$ is a sequence of positive numbers, which necessarily tends to $0$, chosen so that $w_i$ converges (for instance, the $c_i$'s can be normalization constants with respect to an auxiliary Riemannian metric). This situation is depicted in Figure~\ref{fig:proof}.

 By Proposition~\ref{prop:NFOHCauchy}, $(M,g)$ is globally hyperbolic, {so we may choose a convex normal neighborhood $U$ of $p$. Then, there is $\varepsilon > 0$ small enough such that for all $t \in (0,\varepsilon]$ and all $i \in \bN$, we have $p_i, \exp( t v_i) \in U$ and} $p_i \not\ll \exp(t v_i)$. We may also multiply all the $w_i$ by the same constant $c'$, so that $w'_i := c'w_i \to w' := c' w$. Since $s \mapsto \exp(s w)$ is an inextendible causal curve, the NFOH implies that for $c'$ large enough, $q \ll \exp(w')$. On the one hand, by continuity of $\exp$ and openness of $\ll$, this implies that $\phi_i(p_i) \ll \exp(w'_i)$ for all large enough $i$. On the other hand, because the $\phi_i$ are isometries,
 \begin{equation*}
  \exp(w'_i) = \exp(c' c_i D\phi_i(v_i)) = \phi_i(\exp(c' c_i v_i)).
 \end{equation*}
 Since $c_i \to 0$, it holds that $c' c_i \leq \varepsilon$ for all large enough $i$, which implies that $p_i \not\ll \exp(c' c_i v_i)$. But the isometries $\phi_i$ must preserve $\ll$, so we have reached a contradiction. It follows that for any converging sequence of frames $F_i$, the sequence of frames $D \phi_i (F_i)$ must remain in a compact set, concluding the proof.
\end{proof}

\begin{figure}
 \centering
 \includegraphics[width=\textwidth]{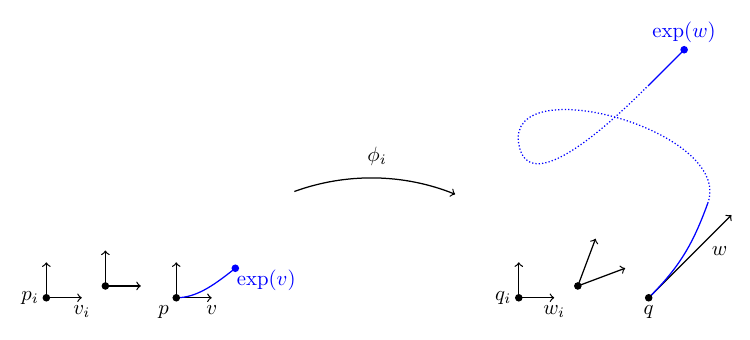}
 \caption{A drawing of the proof of Theorem~\ref{thm:Isomproper}. The situation depicted here leads to a contradiction with the NFOH.}
 \label{fig:proof}
\end{figure}

Corollary~\ref{cor:timeIntro} now arises as a special case of the following theorem.

\begin{thm} \label{thm:time}
 Let $(M,g)$ be a causal spacetime satisfying the NFOH. Suppose that $G$ is a Lie group acting properly on $M$ by time-orientation preserving conformal transformations. Then
 \begin{enumerate}
  \item There exists a Cauchy temporal function $\tau \colon M \to \bR$ such that the action of $G$ on $M$ preserves the {differential} $d\tau$.
  \item There is an induced smooth proper action of $G$ on $\bR$ by translations. This action is, up to scaling, independent of the specific choice of $\tau$.
 \end{enumerate}
\end{thm}

\begin{proof}
 \textit{(i)} By Proposition~\ref{prop:NFOHCauchy} and Theorem~\ref{thm:split}, there is a Cauchy temporal function $\hat\tau \colon M \to \bR$. Let $\mu$ be a left-invariant Haar measure on $G$. If $G$ is compact, then $\mu(G)$ is finite, and hence
 \begin{equation} \label{eq:dtaucmpct}
  d\tau := \int_{G} \phi^* d \tau_0 \ d\mu(\phi).
 \end{equation}
 is a well-defined one-form on $M$. Moreover,
 \begin{equation*}
  \Lambda_p := \{\phi^* d_p\hat\tau \mid \phi \in G \} \subset T^*_pM
 \end{equation*}
 is compact and contained in the interior of the future lightcone. Thus $d_p\tau$ is future-directed timelike, because it lies in the convex hull of $\Lambda_p$. It is also clear that $d\tau$ is $G$-invariant. By Lebesgue's dominated convergence theorem,
 \begin{equation*}
  \int_c \int_{G} \phi^* d \hat\tau(\dot{c}) \ d\mu(\phi) ds = \int_{G} \int_c  \phi^* d \hat\tau(\dot{c}) \  ds d\mu(\phi)
 \end{equation*}
 for every curve $c : [0,1] \to M$. The right hand side vanishes for closed curves, because for every $\phi$, $\phi^* d \hat\tau$ is exact. It follows that $d \tau$ is also exact, and hence indeed it is the differential of a temporal function $\tau$.

 It remains to deal with the case of non-compact $G$, the problem being that the integral \eqref{eq:dtaucmpct} is then infinite. To avoid this, we replace $\hat\tau$ by a suitable function $\tau_0$ with compact support $K \subset M$. Then,
 \begin{equation} \label{eq:dtau}
  d\tau := \int_{G} \phi^* d \tau_0 \ d\mu(\phi) = \int_{G_p} \phi^* d \tau_0 \ d\mu(\phi),
 \end{equation}
 where
 \begin{equation*}
  G_p = \{ \phi \in G \mid \phi^{-1}(p) \in K \}
 \end{equation*}
 is compact in $G$ (by properness of the action), and hence the integral \eqref{eq:dtau} is finite.

 To define $\tau_0$, let $a<b$ and let $f \colon \bR \to \bR$ be a smooth function such that $f(x) = 0$ for $x \leq a$, $f'(x)>0$ for $x \in (a,b)$, and $f(x)=1$ for $x \geq b$. Then $\tau_0 = f \circ \hat\tau$ is a smooth function on $M$ satisfying
 \begin{itemize}
  \item $\tau_0(p) = 0$ for all $p \in J^-(\Sigma_{a})$,
  \item $d \tau_0$ is future-directed timelike on $K := I^+(\Sigma_{a}) \cap I^-(\Sigma_b)$, and
  \item $\tau_0(p) = 1$ for all $p \in J^+(\Sigma_b)$.
 \end{itemize}
 In order for \eqref{eq:dtau} to be everywhere timelike, we need to ensure that every $G$-orbit has a representative in $K$ (because if there is $p$ such that the orbit of $p$ is disjoint from $K$, then $d_p\tau = 0$, which we have to avoid).

 By Lemma~\ref{lem:NOHCauchy}, there are real numbers $a < b < c < d < e$ such that the corresponding $\hat\tau$ level sets $\Sigma_a$, $\Sigma_b$, $\Sigma_c$, $\Sigma_d$, and $\Sigma_e$ are pairwise totally timelike connected. Consider the compact set $A := \hat\tau^{-1}((a,e)) = I^+(\Sigma_a) \cap I^-(\Sigma_e)$ and a point $p \in \Sigma_c \subset A$. By properness of the action, if $G$ is non-compact, then there must be $\phi \in G$ such that $\phi(p) \not\in A$. Since $\Sigma_c$ is a Cauchy surface, so is $S_1 := \phi(\Sigma_c)$, and it follows by achronality that $S_1 \cap \hat\tau^{-1}((b,d)) = \emptyset$. Therefore $\Sigma_c$ and $S_1$ are totally timelike connected (or $S_1$ and $\Sigma_c$, but then we may replace $\phi$ by $\phi^{-1}$, so there is no loss in generality). Moreover, we recursively construct a family of Cauchy surfaces $S_0 := \Sigma_c$, $S_i := \phi^{i}(S_0)$ for all $i \in \bZ$ with $S_{i+1} \subset I^+(S_i)$. We claim that every $G$-orbit contains a representative in the compact set
 \begin{equation} \label{eq:union}
  K_i := J^+(S_i) \cap J^-(S_{i+1}),
 \end{equation}
 for, say, $i=0$ (compactness follows from global hyperbolicity, given that $S_i$ is compact).

 To see that every $G$-orbit contains a representative in $K_0$, first note that $K_i = \phi^i(K_0)$. Hence it suffices to prove that
 \begin{equation*}
  M = \bigcup_{i \in \bZ} K_i.
 \end{equation*}
 Viewing $M = \bR \times S_0$ (as in Theorem~\ref{thm:split}), since the $S_i$ are Cauchy surfaces, each of them is the graph of a function $F_i \colon S_0 \to \bR$. We need to show that $F_i(x) \to \pm \infty$ when $i \to \pm \infty$, for all $x \in S_0$. Suppose that there is $x_0 \in S_0$ such that $F_i(x_0) \not\to \pm \infty$. Since $S_i \in I^-(S_{i+1})$, it follows that $F_i < F_{i+1}$. Thus, without loss of generality, $F_i(x_0) \to C$ for $i \to +\infty$, where $C$ is a constant. Then, the points $x_i := (F_i(x_0),x_0) \in S_i$ converge to $x_\infty = (C,x_0)$, and since the $S_i$ are totally timelike connected, $x_i \ll x_{i+1}$. We can construct a timelike curve $\gamma$ from $x_0$ to $x_\infty$ by concatenating maximizing geodesic segments $\gamma_i$ going from $x_i$ to $x_{i+1}$. Let $d \colon M \times M \to [0,\infty)$ be the Lorentzian distance (which is continuous by global hyperbolicity), and define
 \begin{equation*}
  D_i := \inf \{ d(p,q) \mid p \in S_i, q \in S_{i+1} \}.
 \end{equation*}
 By compactness of the $S_i$ and continuity of $d$, it follows that $D_i > 0$. Moreover, since $d$ is invariant under isometries, $D_i$ is independent of $i$. Now, the Lorentzian length $L(\gamma)$ satisfies
 \begin{equation} \label{eq:length}
  L(\gamma) = \sum_{i=0}^\infty L(\gamma_i) > \sum_{i=0}^\infty D_i = \infty.
 \end{equation}
 It follows that $\gamma$ must be future inextendible, a contradiction.

 It remains to show that $\tau$ is a Cauchy temporal function. We again distinguish between the cases of $G$ compact or non-compact. When $G$ is compact, it follows from \eqref{eq:dtaucmpct} that $\tau$ has, up to an additive constant, the form
 \begin{equation*}
  \tau = \int_G \hat\tau \circ \phi \ d\mu(\phi).
 \end{equation*}
 Let $\gamma \colon \bR \to M$ be an inextendible causal curve (then, so is $\phi \circ \gamma$). Because $\hat\tau$ is Cauchy, we have that
 \begin{equation*}
  \hat\tau \circ \phi \circ \gamma(s) \to \pm \infty
 \end{equation*}
 for every $\phi \in G$. By compactness of $G$, it then follows that $\tau(\gamma(s)) \to \pm \infty$ as $s \to \pm \infty$, implying that $\tau$ is Cauchy. In the case that $G$ is non-compact, we use that the decomposition \eqref{eq:union} implies that $(M,g)$ in fact satisfied the NOH. Indeed, any inextendible causal curve $\gamma$ must traverse all of the $S_i$'s, which are totally timelike connected, and it is easy to conclude that $I^-(\gamma) = M$. Then, it follows from Lemma~\ref{lem:tauCauchy} that $\tau$ must be Cauchy.

 \textit{(ii)} For all $p,q \in M$ and $\phi \in G$, integrating $d\tau$ along a curve from $p$ to $q$ allows us to prove
 \begin{equation*}
  \tau(p) - \tau(q) = \tau(\phi(p)) - \tau(\phi(q)).
 \end{equation*}
 It follows that $c_\phi := \tau(\phi(p)) - \tau(p)$ is independent of $p$. Thus there is an induced action
 \begin{equation*}
  G \times \bR \longrightarrow \bR, \qquad (g,x) \longmapsto x+c_\phi.
 \end{equation*}
 Let $K$ be a compact subset of $\bR$. Because the $\tau$ is Cauchy, by Theorem~\ref{thm:split}, the manifold $M$ splits as a product $M = \bR \times \Sigma$, where the $\bR$ factor is parametrized by $\tau$. It follows that $K' = \tau^{-1}(K)$ is compact in $M$. Therefore $G_K = G_{K'}$ (see Proposition \ref{prop:properactions}), and properness of the action of $G$ on $\bR$ follows from properness of the action on $M$.

 Finally, we show that the action on $\bR$ is, up to scaling, independent of the specific choice of $\tau$, meaning the following: Suppose $\tau_1$, $\tau_2$ are two Cauchy temporal functions with $d\tau_1$, $d\tau_2$ invariant under the action of $G$. Then, for each $\phi \in G$,
 \begin{equation*}
  \tau_1(\phi(p)) - \tau_1(p) = 0 \iff \tau_2(\phi(p)) - \tau_2(p) = 0
 \end{equation*}
 for one, and hence all, $p \in M$. Indeed, suppose that $\tau_1(\phi(p)) - \tau_1(p) = 0$ but $\tau_2(\phi(p)) - \tau_2(p) \neq 0$, and let $p_n := \phi^n(p)$. On the one hand, $p_n$ stays in the compact level set with value $\tau_1(p)$. On the other hand, $\tau_2(p_n) \to \pm \infty$, and hence $p_n$ leaves every compact set, a contradiction.
\end{proof}

\begin{rem}[Comments on the proof of Theorem~\ref{thm:time}]
 {In the case of compact $G$, there cannot be an isometry $\phi$ acting on $\bR$ by a non-trivial translation (otherwise, its powers $\phi^k$ would form a sequence with no convergent subsequence). Therefore $\tau$ itself, and not only $d\tau$, is invariant. A similar construction of an invariant temporal function was given by M\"uller in \cite{Mue}. In the case of non-compact $G$, it follows from the proof that there is an isometry $\phi$ that not only acts by a non-trivial time translation, but even has the property that $p \ll \phi(p)$ for every $p \in M$. In our follow-up work \cite{GaZe}, conformal transformations with this property are called \emph{escaping}, and a similar argument with a fundamental domain $K$ is performed. In the present paper, however, we exploit the fact that we are working with isometries in equation \eqref{eq:length}, and this allows us to show that when there is an escaping isometry $\phi$, the NFOH is automatically promoted to the NOH.}
\end{rem}

The action of $G$ on $\bR$ by translations allows us to prove Corollary~\ref{cor:prodintro}, the first part of which we restate, with added detail, as the following theorem.

\begin{thm} \label{thm:prod}
 Let $(M,g)$ be a causal spacetime satisfying the NFOH, and let $\tau \colon M \to \bR$ be as in Theorem~\ref{thm:time}. Then
 \begin{equation*}
  \Isom^\uparrow(M,g) = L \ltimes N,
 \end{equation*}
 where $N$ is the compact normal subgroup of isometries which leave $\tau$ (and not just $d\tau$) invariant, and $L \cong \{e\}$, $\bZ$ or $\bR$. Moreover, $L$ (up to isomorphism) and $N$ (in the strict sense, as a subgroup of $\Isom^\uparrow(M,g)$) are independent of the choice of $\tau$.
\end{thm}

Based on this result, $N$ can be interpreted as spatial isometries, and $L$ as time-translations. (This motivates the order of the factors in $L \ltimes N$, in analogy to the splitting of the spacetime in Theorem~\ref{thm:split}.) In particular, elements of $N$ acts by a Riemannian isometry on each $\tau$-level set. However, this interpretation should be taken with a grain of salt, as Example~\ref{exam:spiral} below illustrates. The remaining part of the proof is strictly in terms of Lie groups, with no further Lorentzian geometry entering the argument. Hence we phrase it in this generality, and we reverse the order of the factors in the (semi-)direct product, to conform with the more common notation in group theory.

\begin{proof}
 By Theorem~\ref{thm:time}, there is a proper action of $\Isom^\uparrow(M,g)$ on $\bR$ by translations, unique up to scaling. Hence $N$ is independent of $\tau$, and Lemma~\ref{lem:semidirect} below proves the theorem. That $L$ is unique up to isomorphism is because the three options $L \cong \{e\}$, $\bZ$ or $\bR$ imply different topological properties for $\Isom^\uparrow(M,g)$: compact, non-compact with infinitely many connected components, or non-compact with finitely many connected components.
\end{proof}

\begin{lem} \label{lem:semidirect}
 Let $G$ be a Lie group acting smoothly and properly by translations on $\bR$. Then, the elements acting trivially form a compact normal subgroup $N$, and
 \begin{equation*}
  G = N \rtimes L,
 \end{equation*}
 where $L$ is either trivial, $\bZ$, or $\bR$.
\end{lem}

\begin{proof}
 Because the action is by translations, the elements of $G$ that act trivially form a normal subgroup $N$. Since $\{0\} \subset \bR$ is compact and is preserved by the action of $N$, it follows by properness that $N$ is compact. If every element in $G$ acts trivially, then $G = N$, and we are done. It remains to deal with the case when some elements act non-trivially. To that end, define
 \begin{equation*}
  \lambda := \inf\{ c \in (0,\infty) \mid \exists \phi \in G \text{ that acts by } x \mapsto x + c \}
 \end{equation*}
  We distinguish the cases $\lambda >0$ and $\lambda =0$, which correspond to $G = N \rtimes \bZ$ and $G = N \rtimes \bR$, respectively.
 
 \textbf{Case 1.} Suppose that $\lambda > 0$, and let $\phi_i$ be a sequence in $G$ that realizes the infimum. Because the action of $G$ on $\bR$ is proper, and $\phi_i \in G_{[0,\lambda+1]}$ for large $i$, the sequence $\phi_i$ has a convergent subsequence, whose limit $\phi$ acts precisely by $x \mapsto x+\lambda$. Let $\psi \in G$ be any element, acting by $x \mapsto x +c$ for some $c \in \bR$. Then $c$ is an integer multiple of $\lambda$; otherwise, there would be integers $k,l$ such that $c' := \lambda^k c^l \in (0, \lambda)$, and then $\phi^k \psi^l$ would act by $x \mapsto x+c'$, a contradiction. Now, for $j \in \bZ$, define
 \begin{equation*}
 \hat G_j :=\{ \phi \in G \mid \phi \text{ acts by } x \mapsto x + \lambda^j \}.
 \end{equation*}
 By the previous discussion, $G = \bigcup_j \hat G_j$. Similarly to how a Lie group decomposes into its connected components, we have that $\hat G_k = \hat G_0 \phi^k$ and hence every element in $G$ is of the form $\psi \phi^k$ for $\psi \in \hat G_0$ and  $k \in \bZ$. It follows that $G = N \rtimes \bZ$, where $N = \hat G_0$.
 
\textbf{Case 2.}  Suppose that $\lambda = 0$, and let $\phi_i$ be a sequence in $G$ that realizes the infimum. Similarly to before, by properness, the sequence $\phi_i$ has a convergent subsequence, and its limit $\phi$ acts trivially on $\bR$. Denote this subsequence again by $\phi_i$. Then $\psi_i :=\phi^{-1} \phi_i$ is a sequence of elements that act non-trivially and converge to the identity. Because the exponential map is a local diffeomorphism around the identity, it follows that there is a one-parameter subgroup $\Phi \colon \bR \to G$ such that $\Phi(s)$ acts by $x \mapsto x+s$. Let $\psi \in G$ be any element, acting by $x \mapsto x + c$ for some $c \in \bR$. Then $\psi \circ \Phi(-c)$ acts trivially, and we can write $\psi = (\psi \Phi(-c)) \Phi(c)$. It follows that $G = N \rtimes \bR$.
\end{proof}

We end this section by proving the last part of Corollary~\ref{cor:prodintro}, i.e.\ that when $\Isom^{\uparrow}(M,g) \cong \bR \ltimes N$, then the connected component of the identity splits as a direct (and not just semi-direct) product. The relevant lemma about Lie groups below can also be found in \cite[Lem.~5.3]{HMZ}. Its proof requires connectedness, for two reasons: to have a surjective exponential map, allowing us to argue with the Lie algebra, and to be able to pass to the universal covering group.

\begin{lem} \label{lem:direct}
 Consider $N$ a compact, connected Lie group, and let $G := N \rtimes_\Psi L$ with $L \cong \bR$. Then, there is a subgroup $L' \cong \bR$ of $G$ such that $G = N \times L'$.
\end{lem}

\begin{proof}
 We first deal with the case $N = \bT^k \times S$ for $S$ a compact, connected, and semi-simple Lie group. Since $L \cong \bR$, we may view $\Psi \colon L \to \Aut(N)$ as a one-parameter subgroup $s \mapsto \Psi(s)$. Note that $\Aut(N) = \Aut(\bT^k) \times \Aut(S)$, because every automorphism of $N$ must preserve the identity component of the center $Z^0(N)  = \bT^k$. Since $\Aut(\bT^k) = \GL(k,\bZ)$ is discrete, $\Psi(s) \in \Aut^0(S)$ for all $s \in \bR$ (extended trivially to $N$). Note that because $S$ is compact and connected, by \cite[Cor~4.48]{Knapp}, the exponential map $\exp \colon \mathfrak{s} \to S$ is surjective. Therefore, the automorphism $\Psi(s) \in \Aut(S)$ is uniquely determined by $d \Psi(s) \in \Aut(\mathfrak{s})$. Then $s \mapsto d\Psi(s)$ is a one parameter subgroup of $\Aut(\mathfrak{s})$. By \cite[Chap.~I.14]{Knapp}, the Lie algebra of $\Aut(\mathfrak{s})$ is $\Der(\mathfrak{s})= \ad(\mathfrak{s})$.
 Therefore, there exists a $V \in \mathfrak{s}$ such that
 \begin{equation*}
  \frac{d}{ds} \biggr\vert_{s=0} d\Psi(s) = \ad(V) = [V, \cdot ].
 \end{equation*}
 It follows by the above and \cite[Prop.~1.91]{Knapp} that $d\Psi(s) = e^{\ad(sV)} = \Ad(\exp(sV))$, where $e^{(\cdot)} \colon \Der(\mathfrak{s}) \to \Aut(\mathfrak{s})$ is the exponential map, and hence
 \begin{equation} \label{eq:ad1}
  \Psi(s) \exp(X) = \exp(d \Psi(s) X) = \exp(\Ad(\exp(sV)) X)
 \end{equation}
  for all $X \in \mathfrak{n}$ (note that $\exp \colon \mathfrak{n} \to N$ is surjective). On the other hand, since $L \cong \bR$, we can view $L$ as a one-parameter subgroup $s \mapsto \exp(s T)$ for $T \in \mathfrak{g}$, and then, by definition of semi-direct product,
  \begin{equation} \label{eq:ad2}
  \Psi(s) \exp(X) = \exp(sT) \exp(X) \exp(-sT) = \exp(\Ad(\exp(sT)) X).
 \end{equation}
 Because $\exp$ restricted to a neighborhood of the identity is injective, \eqref{eq:ad1} and \eqref{eq:ad2} together imply that $\Ad(\exp(sV)) = \Ad(\exp(sT))$. Hence, the differential $\ad$ of $\Ad$ satisfies $\ad(V) = \ad(T)$, which implies $\ad(T-V) = 0$. Let $L'$ be the one-parameter subgroup $s \mapsto \exp (sT')$, where $T' = T-V$. Then $\Ad(\exp(sT')) = 0$ for all $s$. In other words, elements of $L'$ commute with elements of $N$. Moreover, $\mathfrak{g} = \mathfrak{n} \oplus \mathfrak{l}'$. It follows that $G = N \times L'$, as desired.

 It remains to deal with the case of general $N$. We do this by passing to a covering of the form dealt with in the previous case. By \cite[Prop.~1.124]{Knapp} and \cite[Cor.~4.25]{Knapp}, the Lie algebra $\mathfrak{g}$ of $G$ is of the form
 \begin{equation*}
  \mathfrak{g} = (\bR^k \oplus \mathfrak{s}) \oplus_\psi \mathfrak{l}
 \end{equation*}
 where $\psi$ is the differential of $\Psi$, and $\mathfrak{s}$ is a compact semi-simple Lie algebra. Therefore, there is a compact, simply connected, semi-simple Lie group $S$ with Lie algebra $\mathfrak{s}$. Then $\bR^k \times S$ has Lie algebra $\bR^k \oplus \mathfrak{s}$. By \cite[Thm.~1.125]{Knapp}, there exists a unique map $\Phi \colon L \to \Aut(\bR^k \times S)$ such that
 \begin{equation*}
  \tilde{G} = (\bR^k \times S) \rtimes_\Phi L
 \end{equation*}
 has Lie algebra equal to $\mathfrak{g}$. Because $\tilde{G}$ is simply connected, by \cite[Prop.~1.100]{Knapp}, $\tilde{G}$ must be the universal covering group of $G$. Hence $G = \tilde{G} / \Gamma$ for $\Gamma \subset \tilde{G}$ a discrete central subgroup \cite[Prop.~1.101]{Knapp}. Note that either $L \subset Z(\tilde G)$ (if $\psi$ and hence $\Phi$ are trivial) or $L \cap Z(\tilde G) = \{e\}$ (else). In the first case, we have nothing to prove, so assume that we are in the second one, implying $\Gamma \cap L = \{e\}$. Let $e_S \in S$, $e_L \in L$ denote the identity element, define
 \begin{equation*}
  \hat \Gamma := \Gamma \cap \left( \bR^k \times \{e_S\} \times \{e_L\} \right), \qquad \hat G = \tilde G / \hat \Gamma,
 \end{equation*}
 and denote by $\pi \colon \tilde G \to \hat G$ the quotient map. Moreover, let $p \colon \tilde G \to S$ denote the natural projection. Then $p(\Gamma)$ is a discrete subgroup of $S$, and hence, by compactness of $S$, $p(\Gamma)$ must be finite. Since $\Gamma$ is (with slight abuse of notation) a subgroup of $\hat \Gamma \times p(\Gamma) \times \{e_L\}$, it follows that $\pi(\Gamma) \subset p(\Gamma)$, and hence $\pi(\Gamma)$ is also finite. Therefore
 \begin{equation*}
  \hat G \cong (\bT^k \times S)  \rtimes_\Phi L,
 \end{equation*}
 since otherwise
 \begin{equation*}
  G = \tilde G / \Gamma = \hat G / \pi(\Gamma)
 \end{equation*}
 could not be homeomorphic to $N \times \bR$ for $N$ compact.

 By the first part of the proof, there is an isomorphism $F \colon \hat G \to \bT^k \times S \times L'$, which restricts to the identity on $\bT^k \times S$. Hence $F(\pi(\Gamma)) \cap L' = \{e\}$, and $F(\hat G) / F( \pi (\Gamma)) \cong N \times L'$. Moreover, the isomorphism $F$ descends to the quotients, forming a commutative diagram
 \begin{equation*}
  \begin{tikzcd}
   \hat G \arrow{r}{F} \arrow[swap]{d}{} & F(G) \arrow{d}{} \\%
   G \arrow{r}{}& N \times L'
   \end{tikzcd}
 \end{equation*}
where the vertical arrows are the quotient maps, and the lower horizontal arrow is the isomorphism that proves the lemma.
\end{proof}

\subsection{Examples}

The simplest class of spacetimes where our theorem applies are product spacetimes.

\begin{prop} \label{prop:prod}
 Let $(\Sigma,h)$ be a compact Riemannian manifold, and equip $M = \bR \times \Sigma$ with the product Lorentzian metric $g = -dt^2 + h$. Then
 \begin{equation} \label{eq:Isomprod}
  \Isom^\uparrow(M,g) \cong \bR \times \Isom(\Sigma,h),
 \end{equation}
 where the first factor acts on $M$ by time translations $(t,x) \mapsto (t+c,x)$.
\end{prop}

\begin{proof}
 Clearly, the right hand side of \eqref{eq:Isomprod} is at least a subgroup of $\Isom^\uparrow(M,g)$. We need to prove that there are no additional isometries. Let $\phi \in \Isom^\uparrow(M,g)$ be arbitrary.

 \textbf{Claim.} \textit{For every $x \in \Sigma$, the vertical line $\bR \times \{x\} \subset M$ is preserved by $\phi$.}
 \newline \noindent This is because the vertical lines maximize the distance between all of their points, and are unique with this property, hence must be preserved. Indeed, denote by $\gamma \colon \bR \to M,\ s \mapsto (s,x)$ a parametrized vertical line, and by $\sigma := \phi \circ \gamma$, $\sigma(s) = (t(s),x(s))$ its composition with $\phi$. Because $\phi$ is an isometry and $\gamma$ is a unit-speed geodesic, also $\sigma$ is a unit-speed geodesic. By standard properties of geodesics in semi-Riemannian products, it follows that $\sigma(s) = (a s, x(s))$ for a constant $a > 0$ and a geodesic $x \colon \bR \to \Sigma$ in $(\Sigma,h)$ of constant speed $b = h(\dot x, \dot x)$. Thus
 \begin{equation*}
  1 = g(\dot \sigma, \dot \sigma) = -a^2 + b^2.
 \end{equation*}
 Now consider the sequence of points $p_n := \sigma(n) \in M$ for $n \in \bN$. By compactness of $\Sigma$, a subsequence $(x_{n_k})_k$ of $x_n := x(n) \in \Sigma$ converges to a point $x_\infty \in \Sigma$. Thus, the Lorentzian distance satisfies
 \begin{equation*}
  d_g(p_{n_k},p_{n_l}) = \sqrt{a^2(n_l-n_k)^2- d_h(x_{n_k},x_{n_l})} \longrightarrow \vert n_l-n_k \vert a \quad \text{as } k,l \to \infty,
 \end{equation*}
 where $d_h$ is the Riemannian distance on $\Sigma$. But since $\sigma$ maximizes the Lorentzian length $L_g$,
 \begin{equation*}
  d_g(p_{n_k},p_{n_l}) = L_g(\sigma \vert_{[n_k,n_l]}) = \vert n_l - n_k \vert.
 \end{equation*}
 It follows that $a=1$ and $b = 0$, and hence $x(s)$ is constant, proving the claim.

 To conclude, from the above claim it follows that the coordinate vector field $\partial_t$ is preserved by $\phi$, since it is normalized and coincides with the tangent vector field to the vertical lines. Thus also $g(\partial_t, \cdot) = dt$ is preserved. The proof of Corollary~\ref{cor:prodintro} applied to this specific choice of temporal function then tells us that
 \begin{equation*}
  \Isom^\uparrow(M,g) = L \ltimes N,
 \end{equation*}
 where $N$ is the subgroup of isometries that preserve $t$. Hence $N$ acts by a Riemannian isometry on each slice $\{t\} \times \Sigma$. But because the vertical lines are preserved, $N$ must act by the same isometry on every slice. Therefore $N \cong \Isom(\Sigma,h)$. Moreover, by Theorem~\ref{thm:time}, every $\phi \in \Isom^\uparrow(M,g)$ acts by a translation $t \mapsto t +c$ on the $\bR$ factor in $M = \bR \times \Sigma$, and we can write $\phi = \varphi_c \circ \psi$, for $\varphi_c \colon (t,x) \mapsto (t+c,x)$ and $\psi = \varphi_c^{-1} \circ \phi \in N$. Since all such time-translations $\varphi_c$ are isometries of $(M,g)$, and commute with all spatial isometries, it follows that $\Isom^\uparrow(M,g) \cong \bR \times \Isom(\Sigma,h)$.
\end{proof}

In the above proof, compactness of $\Sigma$ is essential. Indeed, Minkowski spacetime $\bR^{1,n}$ is also a product spacetime, but its isometry group, the Poincar\'e group, includes boosts (which mix space and time), and hence is larger than $\bR \times \OG(n)$. Note also that product spacetimes with compact slices always satisfy the NOH. An example where the NOH is violated and the conclusion of Corollary~\ref{cor:prodintro} fails, is given by de Sitter spacetime $\dS^4$. On the one hand, $\dS^4$ can be represented as the warped product $\dS^4 = \bR^{1,0} \times_{\cosh^2} S^3$, implying that it is globally hyperbolic with compact Cauchy surfaces \cite[pp.~183--184]{BEE}. On the other hand, $\dS^4$ can also be represented as the pseudosphere in Minkowski spacetime $\bR^{1,4}$. It follows that its isometry group is the Lorentz group $\OG^\uparrow(1,4)$ \cite[Prop.~9.8]{ONeill}, which, again because it contains boosts, cannot be written in the form $\bR \ltimes N$ with $N$ compact. While the obvious examples of spacetimes with isometry group $\bR \ltimes N$ are the ones that admit a complete timelike Killing vector field (such as $\partial_t$ in the case of a product spacetime), the following example shows that they are not the only ones.

\begin{exam}[{Figure \ref{fig:spiral}}] \label{exam:spiral}
 Consider on $\bR^2$ the Lorentzian metric $g := f(y) dx^2 - dy^2$ for some function $f \colon \bR \to [0,1)$ with $f(y) = f(y+1)$. Let $M$ be the quotient of $\bR^2$ by the map $(x,y) \mapsto (x+1,y+1)$. The metric $g$ descends to the quotient, which we can model as $\bR \times [0,1]$ with the identifications $(0,y) \sim (1,y+1)$. Hence $M$ is homeomorphic to a cylinder. Because $f<1$, there are no closed causal curves, and in fact, $(M,g)$ is globally hyperbolic with $\{(s,s) \mid s \in [0,1] \} / \sim$ being a compact Cauchy surface. Moreover, $(M,g)$ satisfies the NOH. The coordinate vector field $\partial_x$ is complete, spacelike and Killing, with integral curves escaping to infinity. For generic $f$, it is the only Killing vector field; hence $\Isom^{\uparrow,0}(M,g) \cong \bR$ even though $(M,g)$ does not admit a timelike Killing vector field.
\end{exam}

\begin{figure}
\centering
\begin{minipage}{.5\textwidth}
  \centering
  \includegraphics[height=5cm]{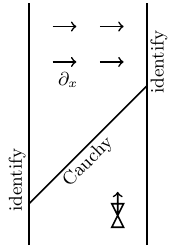}
  \caption{The spacetime of Example \ref{exam:spiral}.}
  \label{fig:spiral}
\end{minipage}%
\begin{minipage}{.5\textwidth}
  \centering
  \includegraphics[height=5cm]{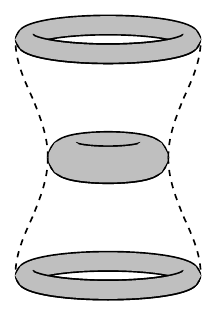}
  \caption{The spacetime of Example \ref{exam:semi}.}
  \label{fig:semi}
\end{minipage}
\end{figure}

Next, we provide a concrete example where $\Isom^\uparrow(M,g) \cong \bZ \ltimes N$. This is generally to be expected when there is some periodic dependence of the metric tensor on a time coordinate.

\begin{exam}[{Figure \ref{fig:semi}}] \label{exam:semi}
 Consider $M := \bR \times S^1 \times S^1$ with metric tensor $g:=-dt^2+ h_t$, where
 \begin{equation*}
  h_t := (\cos^2(t)+1) d\theta^2 + (\sin^2(t)+1)d\varphi^2.
 \end{equation*}
 Note that the level sets of the $t$-coordinate are compact Cauchy surfaces, and that the periodic dependence of $g$ in $t$ implies that $(M,g)$ satisfies the NOH \cite[Thm.~1.6]{GarBoundary}. We are going to show that the subgroup of $\Isom^\uparrow(M,g)$ preserving an invariant time function is
 \begin{equation*}
  N = \OG(2) \times \OG(2),
 \end{equation*}
 where each factor $\OG(2)$ acts by symmetries of one of the circles $S^1$ (e.g.\ rotations and reflections in the $\theta$ or $\varphi$ coordinates). Clearly, these are isometries of $(M,g)$. Let $\Sigma_t := \{t\} \times S^1 \times S^1$. For $t \neq k \pi/2 + \pi/4$,
 \begin{equation*}
  \Isom(\Sigma_t, h_t) = \OG(2) \times \OG(2).
 \end{equation*}
 We can see this because either $\{t\} \times \{\theta\} \times S^1$ or $\{t\} \times S^1 \times \{\varphi\}$ is the image of a closed geodesic of minimal length, for all $\theta$ or all $\varphi$, respectively. Thus, an isometry of $(\Sigma_t,h_t)$ must preserve the splitting of $\Sigma_t$, and hence act on each $S^1$ factor separately.

 By Theorem~\ref{thm:time}, there is a Cauchy temporal function $\tau$ such that $\Isom^\uparrow(M,g)$ acts by translations in $\tau$. Clearly, elements of $\OG(2) \times \OG(2)$ must act by the trivial translation, since otherwise, they would have to send points to infinity, which they cannot. The orbit of each point $p = (t,\theta,\varphi)$ is
 \begin{equation*}
  \{ \phi(p) \mid \phi \in N \} = \Sigma_t.
 \end{equation*}
 Thus, all points on $\Sigma_t$ must lie on the same level set of $\tau$. But since the $\Sigma_t$ are Cauchy surfaces, they must exactly coincide with the $\tau$-level sets, as a proper subset of a Cauchy surface cannot be a Cauchy surface. Moreover, the normal vector field to $\Sigma_t$ is preserved by isometries, and hence an element $\phi \in N$ must restrict to the same map on each $\Sigma_t$. From this and the previous paragraph, it follows that indeed $N$ cannot be larger than $\OG(2) \times \OG(2)$, since $N$ must restrict to an isometry on each level-set $(\Sigma_t,h_t)$.

 Furthermore, let $\phi \colon M \to M$ be the isometry given by
 \begin{equation*}
  \phi(t,\theta,\varphi) := \left(t+\frac{\pi}{2},\varphi,\theta\right),
 \end{equation*}
 i.e., a discrete time translation combined with an interchange of $\theta$ and $\varphi$. Clearly, $\phi$ must act by non-trivial translations of the invariant temporal function $\tau$. We claim that $\{\phi\} \cup N$ generates $\Isom^\uparrow(M,g)$. Suppose not. Then there is an isometry $\psi$ which shifts $\tau$ by a smaller constant than $\phi$ does. Therefore, there must be a point $p \in M$ such that $t(\psi(p)) - t(p) < \frac{\pi}{2}$. But then
 \begin{equation*}
  \psi \vert_{\Sigma_{t(p)}} \colon \Sigma_{t(p)} \longrightarrow \Sigma_{t(\phi(p))}
 \end{equation*}
 is a Riemannian isometry, which is impossible, because the {smaller radius of each torus} $(\Sigma_t,h_t)$ is an invariant, and in this case, they do not match.

 Finally, we argue that $\Isom^\uparrow(M,g) = L \ltimes N$, where $L$ is generated by $\phi$, is a semi-direct product, and is not isomorphic to $\bZ \times N$. Clearly, $\phi$ does not commute with all elements in $N$. Concretely, if $\chi = (\chi_1,\chi_2) \in \OG(2) \times \OG(2)$, then
 \begin{equation*}
  \phi \circ \chi \circ \phi^{-1} = (\chi_2, \chi_1)
 \end{equation*}
 Suppose that there is another isometry $\psi$ such that $\Isom^\uparrow(M,g) = L' \times N$ for $L'$ generated by $\psi$. Necessarily, $\psi = \phi \circ \eta$ for $\eta \in N$. Then
 \begin{equation*}
  \psi \circ \chi \circ \psi^{-1} = \phi \circ \eta \circ \chi \circ \eta^{-1} \circ \phi^{-1}= (\eta_2 \circ \chi_2 \circ \eta_2^{-1}, \eta_1 \circ \chi_1 \circ \eta_1^{-1}).
 \end{equation*}
 One can always find $\chi_1$, $\chi_2$ such that the right hand side does not equal $(\chi_1,\chi_2)$. For instance, take $\chi_1 = e$ and $\chi_2 \neq e$. Hence $\psi$ does not commute with all elements of $N$, a contradiction.
\end{exam}

Example~\ref{exam:semi} not only shows that parts of the conclusion of Corollary~\ref{cor:timeIntro} are sharp. It also shows that our results are, in fact, useful in order to compute isometry groups. Starting from a spacetime with a temporal function invariant under a large group of isometries, we were able to conclude that, in fact, those are all isometries. Compare this with the similar Example~5.5 in \cite{GarBoundary}, where a computation in coordinates was used to show that all Killing vector fields are tangential to the level sets $\Sigma_t$. The following lemma isolates the argument in the second paragraph of Example~\ref{exam:semi}, which does not depend on the specific form of the metric.

\begin{lem} \label{lem:invtime}
 Let $(M,g)$ be a causal spacetime satisfying the NFOH, and let $\tau \colon M \to \bR$ be a Cauchy temporal function. Suppose $G$ is a group acting by isometries on $(M,g)$ such that, for every $t \in \bR$ and one (equivalently, all) points $p \in \Sigma_t$, the orbit $G(p) := \{\phi(p) \mid \phi \in G\}$ equals $\Sigma_t$. Then
 \begin{enumerate}
  \item There is a diffeomorphism $f \colon \bR \to \bR$ such that $d(f \circ \tau)$ is invariant under $\Isom^\uparrow(M,g)$.
  \item The foliation by level sets of $\tau$, and the complementary normal foliation, are preserved by $\Isom^\uparrow(M,g)$.
  \item For $N$ as in Theorem~\ref{thm:prod}, $\phi \in N$ and $t \in \bR$, $\phi$ is uniquely determined by $\phi \vert_{\Sigma_t} \in \Isom(\Sigma_t,g\vert_{\Sigma_t})$.
 \end{enumerate}
\end{lem}

A large class of examples with isometry group $\bZ \ltimes N$ can be obtained through the following construction.

\begin{exam}
 Let $k \geq 2$ and let $\beta \colon \bR \to \GL(k,\bR),\ t \mapsto \beta_t$ be a one-parameter subgroup such that $\beta_1(\bZ^k) = \bZ^k$ but $\beta_1 \neq \operatorname{Id}$. On $\bR \times \bR^k$, consider the equivalence relation
 \begin{equation*}
  (t,v) \sim (t, v') \colon \iff v - v' \in \beta_t(\bZ^k),
 \end{equation*}
 and let $M := (\bR \times \bR^k)/\sim$ be the quotient manifold. The torus $\bT^k = \bR^k / \bZ^k$ acts on $M$ by
 \begin{equation} \label{eq:torusacts}
  u(t,v) = (t, v + \beta_t(u)),
 \end{equation}
 for $u \in \bT^k$, and $\bZ$ acts on $M$ by
 \begin{equation*}
  n(t,v) = (t+n, v).
 \end{equation*}
 for $n \in \bZ$. The $\bT^k$-action is clearly compatible with the equivalence relation $\sim$, and so is the $\bZ$-action, since
 \begin{equation*}
  n(t,v + \beta_t(z)) = (t+n,v + \beta_{t+n}(\beta_{-n}(z))) \sim (t+n,v)
 \end{equation*}
 for all $z \in \bZ^k$, given that, by assumption, $\beta_{-1}(z) \in \bZ^k$. Combining the two actions, we obtain an action of the group $G := \bZ \ltimes_{\beta} \bT^k$ (which is well-defined because $\beta\vert_\bZ$ preserves the integer lattice, and hence defines an automorphism of $\bT^k$). Indeed,
 \begin{equation*}
  un(t,v) = u(t+n,v) = (t+n,v + \beta_{t+n}(u)) = n(t,v+\beta_{t}(\beta_n(u))) = n \beta_n(u) (t,v).
 \end{equation*}

 Because the action of $G$ on $M$ is proper, there exists a $G$-invariant Riemannian metric $g_R$ on $M$. Since $G$ also preserves the foliation $\Sigma$ by $t$-level sets $\Sigma_t$, there is a $G$-invariant splitting of the tangent bundle $TM = T\Sigma \oplus (T\Sigma)^\perp$, where $\perp$ denotes the orthogonal complement with respect to $g_R$. We denote the corresponding orthogonal projection onto $T\Sigma$ by $\pi \colon TM \to T\Sigma$. Moreover, the one-form $dt$ is $G$-invariant, and on each slice $\Sigma_t \cong \bR^n / \beta_t(\bZ_n)$, the Euclidean metric $h_t$ is $G$-invariant as well. Hence
 \begin{equation*}
  g := - dt^2 + \pi^* h_t
 \end{equation*}
 defines a $G$-invariant Lorentzian metric on $M$. As in the previous example, the level sets $\Sigma_t$ are compact Cauchy surfaces, and the periodic dependence of $g$ in $t$ (due to $G$-invariance) implies that $(M,g)$ satisfies the NOH \cite[Thm.~1.6]{GarBoundary}.

 By Lemma~\ref{lem:invtime}, the foliation by $t$-level sets is invariant under $\Isom^\uparrow(M,g)$, and every $\phi \in N$ is uniquely determined by $\phi \vert_{\Sigma_t}$, i.e.\ it acts via the same map on each slice $\Sigma_t$. Since $(\Sigma_t,h_t)\cong \bR^n / \beta_t(\bZ_n)$ is a flat torus, it follows that
 \begin{equation*}
  \Isom(\Sigma_0,h_0) = \bT^k \rtimes F_t,
 \end{equation*}
 where $\bT^k$ acts via \eqref{eq:torusacts}, and  $F_t$ is the finite subgroup of $\OG(k)$ preserving the lattice $\beta_t(\bZ^k)$. Hence
 \begin{equation*}
  N = \bT^k \rtimes F, \quad \text{where } F = \bigcap_{t \in \bR} F_t.
 \end{equation*}
 For a generic choice of $\beta_t$, $F = \{\pm \operatorname{Id}\}$. It follows that $N$ is abelian, and $\Isom^\uparrow(M,g) = \bZ \ltimes N$ cannot be isomorphic to $G' := \bZ \times N$, for the following reason: On the one hand, the center $Z(G')$ contains $\bZ$ as a subgroup. On the other hand, $Z(G)$ is the subgroup of elements in $N$ that commute with $\bZ$, hence compact.
\end{exam}

It remains open if there exist spacetimes $(M,g)$ with $\Isom^\uparrow(M,g) = \bR \ltimes N$ not isomorphic to a direct product. Due to the last part of Corollary~\ref{cor:prodintro}, this is only possible if $N$ has more than one connected component.

\section{Further results} \label{sec:extra}

In this section, we prove a result on uniform Lipschitz continuity of isometries, and give an application thereof to spacetimes with regular cosmological time function (where we do not require the NFOH).

\subsection{Uniform Lipschitzianity of isometries}

In complex analysis, Cauchy's formula
\begin{equation*}
 f^{(n)}(a) = \frac{n!}{2\pi i} \oint_\gamma \frac{f(z)}{(z-a)^{n+1}} dz
\end{equation*}
allows one to compute the derivatives of a holomorphic function by integrating the function over a closed curve. In a similar spirit, we prove a result that establishes a bound on the differential $d\phi$ of an isometry $\phi$ from knowledge of only $\phi$ itself.

Recall that a spacetime $(M,g)$ is called non-totally imprisoning if every (future or past) inextendible causal curve leaves every compact set (this condition is strictly weaker than global hyperbolicity). Furthermore, recall that we can define Lipschitz continuity on a manifold $M$ by choosing an auxiliary complete Riemannian metric on $M$, the specific choice thereof being irrelevant (although the value of the Lipschitz constant does depend on it). As is the case in $\bR^n$, a smooth map $M \to M$ is Lipschitz if and only if its differential is bounded by a constant. We say that an invertible map is bi-Lipschitz if itself and its inverse are both Lipschitz.

\begin{thm} \label{thm:biLip}
 Let $(M,g)$ be a non-totally imprisoning spacetime, and let $U \subset M$ be an open subset. Suppose that $G \subset \Isom^\uparrow(M,g)$ is a subset of isometries such that there exists a compact $K' \subset M$ with $\phi(p) \in K'$ for all $p \in U$ and all $\phi \in G$. Then $G$ is uniformly Lipschitz on every compact $K \subset U$. If $G$ is a subgroup, it is uniformly bi-Lipschitz.
\end{thm}

\begin{proof}
 Let $K \subset U$ be compact, and suppose that the assertion of the theorem is not true. Then there exists a sequence of points $p_i \in K$ and group elements $\phi_i \in G$ such that $D_{p_i} \phi_i$ diverges. By compactness of $K$ and $K'$, we may assume that $p_i \to p$ and $\phi_i(p_i) \to q$. Choose a sequence of orthonormal frames $F_i$ at $p_i$ that converge to an orthonormal frame $F$ at $p$. If $D_{p_i} \phi_i (F_i)$ converges to an orthonormal frame at $q$, then by Theorem~\ref{Thm:OM}, $\phi_i$ converges in $\Isom^\uparrow(M,g)$ in the $C^1$ topology (which on $\Isom^\uparrow(M,g)$ is equivalent to the $C^0$ topology), so $D_{p_i} \phi_i$ must converge as well, contrary to our assumption.

 Let $T_i$ denote the unit timelike vector of the frame $F_i$. We conclude from the above that the sequence of orthonormal frames $D_{p_i} \phi_i (F_i)$ cannot converge. Thus there must be a subsequence that becomes degenerate in the limit, in the sense that the unit-timelike vectors $V_i = D_{p_i} \phi_i (T_i)$ of the frames $D_{p_i} \phi_i (F_i)$ tend (up to scaling) to a null vector $V$. Consider the sequence of curves
 \begin{equation*}
  \gamma_i \colon [0,\delta] \to M, \quad s \mapsto \exp(s T_i),
 \end{equation*}
 where $\exp$ denotes the exponential map with respect to the Lorentzian metric $g$. For $\delta$ small enough and all $i$ large enough, the image of $\gamma_i$ lies in $U$ (we just need a $\delta$ such that the image of $s \mapsto \exp(s T)$ for the limit vector $T = \lim T_i$ stays in $U$). Next, consider a new sequence of curves $\sigma_i \colon [0,\delta] \to M$ given by
 \begin{equation*}
  \sigma_i(s) := \phi_i \circ \gamma_i(s) = \exp(s V_i).
 \end{equation*}
 On the one hand, because the image of each $\gamma_i$ lies in $U$, we have that the image of each $\sigma_i$ lies in $K'$. On the other hand, if we reparametrize the $\sigma_i$ with respect to $h$-arclength, then by the Lorentzian limit curve theorem \cite[Prop.~3.31]{BEE}, $\sigma_i$ converges (up to a subsequence) to a limit causal geodesic $\sigma$ with image contained in $K'$. But up to reparametrization, $\sigma$ can be nothing else than the null geodesic
 \begin{equation*}
  \sigma \colon [0,\infty) \to M, \quad s \mapsto \exp(sV),
 \end{equation*}
 where the domain of definition is $[0,\infty)$ because $\Vert V_i \Vert_h \to \infty$. We reach a contradiction to non-total imprisonment, because then $\sigma$ would be an inextendible causal curve contained in the compact set $K'$.

 It follows that $\Vert D \phi \Vert_h$ must be bounded on $K \subset U$, and the bound can be chosen uniform in $\phi \in G$. If $G$ is a subgroup, then for all $\phi \in G$, the same bound also applies to $\phi^{-1} \in G$. This concludes the proof.
\end{proof}

\subsection{Spacetimes with a regular cosmological time function}

While in Section \ref{sec:Isomproper} we have exploited the NOH to construct an invariant (up to translations) time function, there are also cases where such a function is already naturally present. One example is the cosmological time function
\begin{equation*}
 \tau(p) := \sup \left\{ L_g(\gamma) \mid \gamma \colon [0,1) \to \infty \text{ past-directed causal with } \gamma(0)=p \right\},
\end{equation*}
where $L_g(\gamma)$ denotes the Lorentzian length of $\gamma$. Andersson, Galloway and Howard call $\tau$ regular when it is finite and $\tau \to 0$ along every past-inextendible causal curve. This regularity assumption implies a number of nice properties for $\tau$ and $(M,g)$, including that $\tau$ is indeed a continuous time function, and that $(M,g)$ is globally hyperbolic \cite{AGHCosmo}. It also has the following implication on the isometry group.

\begin{thm} \label{thm:cosmo}
 Let $(M,g)$ be a spacetime with compact Cauchy surfaces. If the cosmological time function $\tau$ is regular, then $\Isom^\uparrow(M,g)$ is compact.
\end{thm}

\begin{proof}
Let $\Sigma_1$ be a Cauchy surface in $M$. By continuity and compactness, $\tau$ attains a minimum value $\varepsilon$ on $\Sigma_1$. Let $\Sigma_2 := \tau^{-1}(\varepsilon/2)$. By construction, $\Sigma_2$ lies in the past of $\Sigma_1$. Choose any smooth past-directed timelike vector field $X$ on $(M,g)$ (e.g.\ the time-orientation vector field). By regularity of $\tau$, every flowline of $X$ must intersect $\Sigma_2$. Hence, the flow of $X$ provides a homeomorphism $\Sigma_1 \to \Sigma_2$, and we conclude that $\Sigma_2$ is compact. Since, by construction, $\tau$ is invariant under the action of $\Isom^\uparrow(M,g)$, it follows that for every $\phi \in \Isom^\uparrow(M,g)$,
 \begin{equation*}
  \phi\left( \Sigma_2 \right) \cap \Sigma_2 \neq \emptyset.
 \end{equation*}
 Compactness of $\Isom^\uparrow(M,g)$ now follows by Theorem~\ref{Thm:OM} if we show that for every sequence of points $p_i \in \Sigma_2$ and isometries $\phi_i$ such that $p_i$ and $\phi_i(p_i)$ converge, there is a converging sequence of orthonormal frames $F_i$ at $p_i$ such that a subsequence of $D \phi_i (F_i)$ converges. The latter holds because $\Isom^\uparrow(M,g)$ is bi-Lipschitz, hence $D \phi_i (F_i)$ stays in a compact set of $\OG(M)$. Bi-Lipschitzianity follows from Theorem~\ref{thm:biLip} applied to $G = \Isom^\uparrow(M,g)$, $U = \tau^{-1}((a,b))$, $K = \Sigma_2$, and $K' = \tau^{-1}([a,b])$ for $0 < a < \varepsilon/2 < b < \epsilon$.
\end{proof}

It is a consequence of the above proof (combined with \cite[Cor.~1]{GalCauchy}) that the level sets of the cosmological time function for small values (smaller than $\varepsilon$) are Cauchy surfaces. This had previously been proven in \cite{GalGar} for all level sets, under the additional assumption that the spacetime is future timelike geodesically complete. Theorem~\ref{thm:cosmo} also applies to the regular cosmological volume functions introduced in \cite[Sec.~6]{GarVol}, as these are also, by construction, invariant under isometries.

We end this section with a class of examples where Theorem~\ref{thm:cosmo} is applicable, while Corollary~\ref{cor:prodintro} is not, because the NOH is not satisfied.

\begin{exam} \label{exam:cone}
 Let $ \mathrm{Co}^{1, n}$ be the (solid) lightcone of the Minkowski spacetime $\bR^{1, n}$ (i.e., the chronological future $I^+(0)$ of $0$). It is foliated by homothetics of hyperbolic space $\mathbb H^n$ (see Figure~\ref{fig:cone}). Let $\Gamma $ be a co-compact (torsion-free) lattice in the Lorentz group $\mathrm{SO}(1, n)$, meaning that $\mathbb H^n / \Gamma$ is a closed hyperbolic manifold. The quotient $M_\Gamma =\mathrm{Co}^{1, n} / \Gamma$ is globally hyperbolic with compact Cauchy surfaces. Furthermore, it is flat, i.e., locally isometric to Minkowski. In fact, $M_\Gamma$ is a warped product $M_\Gamma = \mathbb R^+ \times_w S$, with $S = \mathbb H^n / \Gamma$ and warping function $w(r) = r^2$ (as in the case of polar coordinates in Euclidean spaces). The radius function (i.e., the projection onto $\bR^+$) turns out to be the cosmological time of $M_\Gamma$. Finally, $M_\Gamma$ satisfies the NOH.

 Now, it turns out that  besides this exact solution, there are deformations of it, enjoying the same properties, except for the NOH. For this, let $\Gamma^\prime$ be an affine deformation of $\Gamma$, that is, $\Gamma^\prime$ is a subgroup of the Poincar\'e group $\mathrm{Poin}^{1, n}$, such that every $\gamma^\prime \in \Gamma^\prime$ is of the form $x \to \gamma(x) + t_\gamma$, where $\gamma \in \Gamma$, and $t_\gamma$ is a translation vector. It happens that this deformation can be followed by a deformation of $ \mathrm{Co}^{1, n}$  which becomes a convex globally hyperbolic domain $\mathrm D_{\Gamma^\prime}$ in Minkowski, whose quotient by $\Gamma^\prime$ is a globally hyperbolic spacetime $M_{\Gamma^\prime}$ with compact Cauchy surfaces. It is regular from the cosmological time function point of view, but does not satisfy NOH (unless the deformation is trivial). Its horizon is complicated, and the initial singularity becomes like a $1$-dimensional space, in the case $n = 2$, contrary to the big-bang case of $\mathrm{Co}^{1, n}$. See \cite{Bon,Mess} for more details.
\end{exam}

\begin{figure}
 \centering\includegraphics{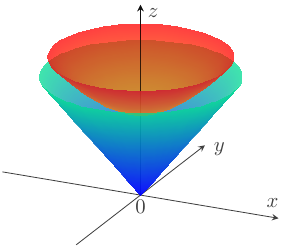}
 \label{fig:cone}
 \caption{The cone $\mathrm{Co}^{1, n}$ from Example~\ref{exam:cone} (in blue), with one of the hyperboloidal sheets depicted in red.}
\end{figure}

\bibliographystyle{abbrv}
\bibliography{refs.bib}

\end{document}